\documentclass[11pt,fleqn]{article}

\usepackage{amsmath,amssymb,amsthm}
\usepackage{geometry} 
\usepackage[pdfpagemode=UseNone,pdfstartview=FitH]{hyperref}

\newcommand{\bdry}[1]{\partial #1}
\newcommand{\D}{{\cal D}}
\newcommand{\dint}{\ds{\int}}
\newcommand{\ds}[1]{\displaystyle #1}
\newcommand{\eps}{\varepsilon}
\newcommand{\norm}[2][]{\left\|#2\right\|_{#1}}
\renewcommand{\o}{\text{o}}
\newcommand{\PS}[1]{$(\text{PS})_{#1}$}
\newcommand{\pnorm}[2][]{\if #1'' \left|#2\right|_p \else \left|#2\right|_{#1} \fi}
\newcommand{\R}{\mathbb R}
\newcommand{\seq}[1]{\left(#1\right)}
\newcommand{\set}[1]{\left\{#1\right\}}
\newcommand{\Z}{\mathbb Z}

\DeclareMathOperator{\divg}{div}

\newenvironment{enumroman}{\begin{enumerate}

}{\end{enumerate}}

\newtheorem{theorem}{Theorem}

\theoremstyle{remark}
\newtheorem{remark}[theorem]{Remark}

\title{\bf On the existence of ground state solutions to critical growth problems nonresonant at zero}
\author{\bf Kanishka Perera\\
Department of Mathematical Sciences\\
Florida Institute of Technology\\
Melbourne, FL 32901, USA\\
\em kperera@fit.edu}
\date{}

\begin{document}

\maketitle

\begin{abstract}
We prove the existence of ground state solutions to critical growth $p$-Laplacian and fractional $p$-Laplacian problems that are nonresonant at zero.
\end{abstract}

\bigskip \bigskip

Consider the problem
\begin{equation} \label{1}
\left\{\begin{aligned}
- \Delta_p u & = \lambda\, |u|^{p-2}\, u + |u|^{p^\ast - 2}\, u && \text{in } \Omega\\[10pt]
u & = 0 && \text{on } \bdry{\Omega},
\end{aligned}\right.
\end{equation}
where $\Omega$ is a bounded domain in $\R^N$, $1 < p < N$, $\Delta_p u = \divg (|\nabla u|^{p-2}\, \nabla u)$ is the $p$-Laplacian of $u$, $\lambda \in \R$, and $p^\ast = Np/(N - p)$ is the critical Sobolev exponent. Solutions of this problem coincide with critical points of the $C^1$-functional
\[
E(u) = \frac{1}{p} \int_\Omega |\nabla u|^p\, dx - \frac{\lambda}{p} \int_\Omega |u|^p\, dx - \frac{1}{p^\ast} \int_\Omega |u|^{p^\ast} dx, \quad u \in W^{1,p}_0(\Omega).
\]
Let $K = \big\{u \in W^{1,p}_0(\Omega) \setminus \set{0} : E'(u) = 0\big\}$ be the set of nontrivial critical points of $E$ and set
\[
c = \inf_{u \in K}\, E(u).
\]
Recall that $u_0 \in K$ is called a ground state solution if $E(u_0) = c$. For each $u \in K$,
\[
E(u) = E(u) - \frac{1}{p^\ast}\, E'(u)\, u = \frac{1}{N} \int_\Omega |u|^{p^\ast} dx > 0,
\]
so $c \ge 0$, and $c > 0$ if there is a ground state solution. Let
\[
S = \inf_{u \in \D^{1,p}(\R^N) \setminus \set{0}}\, \frac{\dint_{\R^N} |\nabla u|^p\, dx}{\left(\dint_{\R^N} |u|^{p^\ast} dx\right)^{p/p^\ast}}
\]
be the best Sobolev constant. Denote by $\sigma(- \Delta_p)$ the Dirichlet spectrum of $- \Delta_p$ in $\Omega$ consisting of those $\lambda \in \R$ for which the eigenvalue problem
\begin{equation} \label{2}
\left\{\begin{aligned}
- \Delta_p u & = \lambda\, |u|^{p-2}\, u && \text{in } \Omega\\[10pt]
u & = 0 && \text{on } \bdry{\Omega}
\end{aligned}\right.
\end{equation}
has a nontrivial solution. We have the following theorem.

\begin{theorem} \label{Theorem 1}
If problem \eqref{1} has a nontrivial solution $u$ with
\begin{equation} \label{3}
E(u) < \frac{1}{N}\, S^{N/p}
\end{equation}
and $\lambda \notin \sigma(- \Delta_p)$, then it has a ground state solution.
\end{theorem}

\begin{proof}
Let $\seq{u_j} \subset K$ be a minimizing sequence for $c$. Then $\seq{u_j}$ is a \PS{c} sequence for $E$. Since problem \eqref{1} has a nontrivial solution satisfying \eqref{3}, $c < S^{N/p}/N$. So $E$ satisfies the \PS{c} condition (see Guedda and V{\'e}ron \cite[Theorem 3.4]{MR1009077}). Hence a renamed subsequence of $\seq{u_j}$ converges to a critical point $u_0$ of $E$ with $E(u_0) = c$. We claim that $u_0$ is nontrivial and hence a ground state solution of problem \eqref{1}. To see this, suppose $u_0 = 0$. Then $\rho_j := \norm{u_j} \to 0$. Let $\tilde{u}_j = u_j/\rho_j$. Since $\norm{\tilde{u}_j} = 1$, a renamed subsequence of $\seq{\tilde{u}_j}$ converges to some $\tilde{u}$ weakly in $W^{1,p}_0(\Omega)$, strongly in $L^p(\Omega)$, and a.e.\! in $\Omega$. Since $E'(u_j) = 0$,
\[
\int_\Omega |\nabla u_j|^{p-2}\, \nabla u_j \cdot \nabla v\, dx = \lambda \int_\Omega |u_j|^{p-2}\, u_j v\, dx + \int_\Omega |u_j|^{p^\ast - 2}\, u_j v\, dx \quad \forall v \in W^{1,p}_0(\Omega),
\]
and dividing this by $\rho_j^{p-1}$ gives
\begin{equation} \label{4}
\int_\Omega |\nabla \tilde{u}_j|^{p-2}\, \nabla \tilde{u}_j \cdot \nabla v\, dx = \lambda \int_\Omega |\tilde{u}_j|^{p-2}\, \tilde{u}_j v\, dx + \o(\norm{v}) \quad \forall v \in W^{1,p}_0(\Omega).
\end{equation}
Passing to the limit in \eqref{4} gives
\[
\int_\Omega |\nabla \tilde{u}|^{p-2}\, \nabla \tilde{u} \cdot \nabla v\, dx = \lambda \int_\Omega |\tilde{u}|^{p-2}\, \tilde{u} v\, dx \quad \forall v \in W^{1,p}_0(\Omega),
\]
so $\tilde{u}$ is a weak solution of \eqref{2}. Taking $v = \tilde{u}_j$ in \eqref{4} and passing to the limit shows that $\lambda \int_\Omega |\tilde{u}|^p\, dx = 1$, so $\tilde{u}$ is nontrivial. This contradicts the assumption that $\lambda \notin \sigma(- \Delta_p)$ and completes the proof.
\end{proof}

Combining this theorem with the existence results in Garc{\'{\i}}a Azorero and Peral Alonso \cite{MR912211}, Egnell \cite{MR956567}, Guedda and V{\'e}ron \cite{MR1009077}, Arioli and Gazzola \cite{MR1741848}, and Degiovanni and Lancelotti \cite{MR2514055} gives us the following theorem for the case $N \ge p^2$.

\begin{theorem} \label{Theorem 2}
If $N \ge p^2$ and $\lambda \in (0,\infty) \setminus \sigma(- \Delta_p)$, then problem \eqref{1} has a ground state solution.
\end{theorem}

For $N < p^2$, combining Theorem \ref{Theorem 1} with Perera et al.\! \cite[Corollary 1.2]{MR3469053} gives the following theorem, where $\seq{\lambda_k} \subset \sigma(- \Delta_p)$ is the sequence of eigenvalues based on the $\Z_2$-cohomological index introduced in Perera \cite{MR1998432} and $|\cdot|$ denotes the Lebesgue measure in $\R^N$.

\begin{theorem} \label{Theorem 3}
If $N < p^2$ and
\[
\lambda \in \bigcup_{k=1}^\infty \Big(\lambda_k - \frac{S}{|\Omega|^{p/N}},\lambda_k\Big) \setminus \sigma(- \Delta_p),
\]
then problem \eqref{1} has a ground state solution.
\end{theorem}

\begin{remark}
In the semilinear case $p = 2$, Theorem \ref{Theorem 2} was proved in Szulkin et al.\! \cite{MR2553063} using a Nehari-Pankov manifold approach, and Theorems \ref{Theorem 1} and \ref{Theorem 3} were proved in Chen et al.\! \cite{MR2926297} using a more direct approach. Moreover, they allow $\lambda$ to be an eigenvalue when $N \ge 5$. However, their proofs are strongly dependent on the fact that $H^1_0(\Omega)$ splits into the direct sum of its subspaces spanned by the eigenfunctions of the Laplacian that correspond to eigenvalues that are less than or equal to $\lambda$ and those that are greater than $\lambda$. Those proofs do not extend to the $p$-Laplacian since it is a nonlinear operator and hence has no linear eigenspaces.
\end{remark}

\begin{remark}
We conjecture that the assumption $\lambda \notin \sigma(- \Delta_p)$ can be removed from Theorems \ref{Theorem 1} and \ref{Theorem 2} when $N^2/(N + 1) > p^2$.
\end{remark}

Our argument can be easily adapted to obtain ground state solutions of other types of critical growth problems as well. For example, consider the nonlocal problem
\begin{equation} \label{5}
\left\{\begin{aligned}
(- \Delta)_p^s\, u & = \lambda\, |u|^{p-2}\, u + |u|^{p_s^\ast - 2}\, u && \text{in } \Omega\\[10pt]
u & = 0 && \text{in } \R^N \setminus \Omega,
\end{aligned}\right.
\end{equation}
where $\Omega$ is a bounded domain in $\R^N$ with Lipschitz boundary, $s \in (0,1)$, $1 < p < N/s$, $(- \Delta)_p^s$ is the fractional $p$-Laplacian operator defined on smooth functions by
\[
(- \Delta)_p^s\, u(x) = 2 \lim_{\eps \searrow 0} \int_{\R^N \setminus B_\eps(x)} \frac{|u(x) - u(y)|^{p-2}\, (u(x) - u(y))}{|x - y|^{N+sp}}\, dy, \quad x \in \R^N,
\]
$\lambda \in \R$, and $p_s^\ast = Np/(N - sp)$ is the fractional critical Sobolev exponent. Let $\pnorm{\cdot}$ denote the norm in $L^p(\R^N)$, let
\[
[u]_{s,p} = \left(\int_{\R^{2N}} \frac{|u(x) - u(y)|^p}{|x - y|^{N+sp}}\, dx dy\right)^{1/p}
\]
be the Gagliardo seminorm of a measurable function $u : \R^N \to \R$, and let
\[
W^{s,p}(\R^N) = \set{u \in L^p(\R^N) : [u]_{s,p} < \infty}
\]
be the fractional Sobolev space endowed with the norm
\[
\norm[s,p]{u} = \left(\pnorm{u}^p + [u]_{s,p}^p\right)^{1/p}.
\]
We work in the closed linear subspace
\[
W^{s,p}_0(\Omega) = \set{u \in W^{s,p}(\R^N) : u = 0 \text{ a.e.\! in } \R^N \setminus \Omega}
\]
equivalently renormed by setting $\norm{\cdot} = [\cdot]_{s,p}$. Solutions of problem \eqref{5} coincide with critical points of the $C^1$-functional
\[
E_s(u) = \frac{1}{p} \int_{\R^{2N}} \frac{|u(x) - u(y)|^p}{|x - y|^{N+sp}}\, dx dy - \frac{\lambda}{p} \int_\Omega |u|^p\, dx - \frac{1}{p_s^\ast} \int_\Omega |u|^{p_s^\ast}\, dx, \quad u \in W^{s,p}_0(\Omega).
\]
As before, a ground state is a least energy nontrivial solution. Let
\[
\dot{W}^{s,p}(\R^N) = \set{u \in L^{p_s^\ast}(\R^N) : [u]_{s,p} < \infty}
\]
endowed with the norm $\norm{\cdot}$ and let
\[
S = \inf_{u \in \dot{W}^{s,p}(\R^N) \setminus \set{0}}\, \frac{\dint_{\R^{2N}} \frac{|u(x) - u(y)|^p}{|x - y|^{N+sp}}\, dx dy}{\left(\dint_{\R^N} |u|^{p_s^\ast}\, dx\right)^{p/p_s^\ast}}
\]
be the best fractional Sobolev constant. Denote by $\sigma((- \Delta)_p^s)$ the Dirichlet spectrum of $(- \Delta)_p^s$ in $\Omega$ consisting of those $\lambda \in \R$ for which the eigenvalue problem
\[
\left\{\begin{aligned}
(- \Delta)_p^s\, u & = \lambda\, |u|^{p-2}\, u && \text{in } \Omega\\[10pt]
u & = 0 && \text{in } \R^N \setminus \Omega
\end{aligned}\right.
\]
has a nontrivial solution. Following theorem can be proved arguing as in the proof of Theorem \ref{Theorem 1}.

\begin{theorem} \label{Theorem 4}
If problem \eqref{5} has a nontrivial solution $u$ with
\[
E_s(u) < \frac{s}{N}\, S^{N/sp}
\]
and $\lambda \notin \sigma((- \Delta)_p^s)$, then it has a ground state solution.
\end{theorem}

Combining this theorem with the existence results in Mosconi et al.\! \cite{MR3530213} and Perera et al.\! \cite{MR3458311} gives us the following theorem, where $\seq{\lambda_k} \subset \sigma((- \Delta)_p^s)$ is the sequence of eigenvalues based on the $\Z_2$-cohomological index.

\begin{theorem} \label{Theorem 5}
Problem \eqref{5} has a ground state solution in each of the following cases:
\begin{enumroman}
\item $N > sp^2$ and $\lambda \in (0,\infty) \setminus \sigma((- \Delta)_p^s)$,
\item $N = sp^2$ and $\lambda \in (0,\lambda_1)$,
\item $N \le sp^2$ and
    \[
    \lambda \in \bigcup_{k=1}^\infty \Big(\lambda_k - \frac{S}{|\Omega|^{sp/N}},\lambda_k\Big) \setminus \sigma((- \Delta)_p^s).
    \]
\end{enumroman}
\end{theorem}

\begin{remark}
Theorems \ref{Theorem 4} and \ref{Theorem 5} are new even in the semilinear case $p = 2$.
\end{remark}

\begin{remark}
We conjecture that problem \eqref{5} has a ground state solution for all $\lambda > 0$ when $N^2/(N + s) > sp^2$.
\end{remark}

{\small \def\cdprime{$''$}
}

\end{document}